\documentclass[12pt,a4paper]{article}
\usepackage{amsmath, amssymb, theorem, latexsym}
\usepackage{color}
\newtheorem{theorem}{Theorem}[section]

\newtheorem{proposition}[theorem]{Proposition}
\newtheorem{definition}[theorem]{Definition}
\newtheorem{conjecture}[theorem]{Conjecture}

\newtheorem{remark}[theorem]{Remark}

\allowdisplaybreaks

\newcommand\finbox{~\hfill$\Box$}%
\newenvironment{proof}[1][Proof]{\noindent\textbf{#1. }{}}{}%
\def\al {{\alpha}}
\def\be {{\beta}}

\def\la {{\lambda}}

\def \rz { {\mathbb R}}
\def\Af {{\mathfrak A}}
\def\af {{\frak a}}

\def\mf {{\frak m}}
\def \nz {{\mathbb N}}
\def \rz {{\mathbb R}}

\def \rz {{\mathbb R}}

\numberwithin{equation}{section}
\newcommand \clb{\color{blue}}

\newcommand \clg{\color{green}}
\begin{document}
%\null\vspace{4cm}

{\centering
\bfseries
{\Large {On maximal multiplicities for Hamiltonians with separable variables}
\\
\noindent
\par
\mdseries
\scshape
\small
B. Helffer$^{*,**}$\\
T. Hoffmann-Ostenhof$^{***}$ \\
P. Marquetand$^{***}$\\
\par
\upshape
Laboratoire de Math\'ematique Jean Leray, Univ. Nantes $^*$\\
Laboratoire de Math\'ematiques d'Orsay, Univ Paris-Sud and CNRS $^{**}$\\
Institut f\"ur Theoretische Chemie, Universit\"at Wien $^{***}$\\~\\
}
}

\begin{abstract}
For $N\in \nz^*:=\mathbb N \setminus \{0\}$, we consider the collection $\mathfrak M(N)$    of all the  $N$ rows, for which, for $n=1,\cdots,N$,  the $n-th$ row consists of an increasing 
 sequence $(a_j^n)_j$  of  real numbers.  For $\Af \in \mathfrak M(N)$, we define its  spectrum  $\sigma(\Af)$ 
by $\sigma(\Af)=\{\la\in \mathbb R \;|\; \la=\sum_{n=1}^Na_{j_n}^n\}\,,$ 
where $(j_1,j_2,\dots,j_N)\in (\nz^*)^N$. This spectrum is discrete and consists of  an infinite sequence that can be ordered 
as a strictly  increasing sequence $\lambda_k(\Af)$. For $\lambda \in \sigma (\Af)$ we denote by $m(\lambda,\Af) $ the number of representations of such a $\la$, hence the multiplicity 
of $\la$.\\  In this paper we investigate for  given $N\in \mathbb N^*$ and $k\in \nz^*$ the highest possible multiplicity (denoted by  $\mf_k(N)$) of $\la_k(\Af)$ for $\Af \in  \mathfrak M(N)$. We give the exact result for $N=2$ and  for $N=3$  prove a lower bound which appears, according to numerical experiments, as  a "good" conjecture. For the general case, we give examples demonstrating that the problem is quite difficult. \\
This  problem  is equivalent to the analogue eigenvalue multiplicity questions for Schr\"odinger operators 
describing a system of N non-interacting one-dimensional particles. 

\end{abstract}
\newpage 
\section{The general problem}
 The  motivation for this paper  is 
the spectral problem  for the operator in $\mathbb R^N$  ($N\in \mathbb N^*:=\mathbb N \setminus \{0\}$).
$$
H_N:= \mathfrak h_1 \otimes I \otimes\cdots \otimes I + I\otimes  \mathfrak h_2 \otimes I \otimes\cdots \otimes I + \cdots +  I \otimes I \otimes\cdots \otimes \mathfrak h_N\,,
$$
see \cite{ReSi} for the notation where we identify $L^2(\mathbb R^N)$ and $ L^2(\mathbb R)\otimes \cdots\otimes L^2(\mathbb R)$.\\
In other words, we consider
\begin{equation}
H_N=\sum_{i=1}^N \mathfrak h_i (x_j,\partial_{x_j})\,,
\end{equation}
where, for each $i$, $$\mathfrak h_i (t,\frac{d}{dt}) = - \frac{d^2}{dt^2} + v_i(t) $$ is a $(1D)$- operator has discrete spectrum: $\sigma(H_i)=\{a^i_k\}_{k=1,2,3,\dots, s,s+1,\dots}$. Note that the sequence $\{k\}$ could be a priori  finite or infinite but we will mainly discuss the infinite case.

 Some classical result  says  that given a finite sequence of numbers,
say $\la_1<\la_2,\dots  <\la_K$  then there is a potential $v$ such that $-\frac{d^2}{dt^2} +v$ in $(1D)$ has those
$\la_i$ as the first $K$ eigenvalues. $K$ is arbitrary, but finite. This can be extended to the case of a bounded countable sequence of eigenvalues (see \cite{GKZ} and references therein).\\
Another classical result (1987) is due to  Y. Colin de Verdi\`ere (see \cite{CdV}). He proves that, in dimension $n\geq 3$, for a given $p$ and a finite sequence $0 < \lambda_2 \leq \cdots \leq \lambda_p\,$, 
 there exists $(M,g)$ smooth with first $p$ eigenvalues of $\Delta$ equal $0,\lambda_2,\dots, \lambda_p\,$. \\
 
Let us describe our problem. 
Let $N\geq 1$ be an integer and consider $N$ sequences $\af^n:=\{a^n_i\}_{i=1}^\infty$ ($n=1,\dots, N$)
and write these $\af^n$ as rows in a matrix $\Af $,
\begin{equation}\label{AfN}
\Af =\left(
\begin{matrix}
a_1^1&a^1_2&a^1_3&\dots& a^1_s&a^1_{s+1}&\dots \\
a^2_1&a^2_2&a^2_3&\dots &a^2_s&a^2_{s+1}&\dots \\
\dots&\dots&\dots&\dots&\dots&\dots&\dots\\
\dots&\dots&\dots&\dots&\dots&\dots&\dots\\
a^{N-1}_1&a^{N-1}_2&a_3^{N-1}&\dots& a_s^{N-1}&a_{s+1}^{N-1}& \dots \\
a^N_1&a^N_2&a^N_3&\dots&a^N_s&a^N_{s+1}&\dots 
\end{matrix}
\right.\,,
\end{equation}
where we assume that the $a_i^n$ are real-valued and satisfy for each $n$ 
\begin{equation}\label{ineq}
a_i^n<a_{i+1}^n\,. 
\end{equation}
We define the spectrum of $\Af $ by 
\begin{equation}\label{sigmaN}
 \sigma(\Af )=\{\la \in \rz\:|\: \la =\sum_{n=1}^Na^n_{j_n}\}\,,
\end{equation}
where  $(j_1,j_2,\dots, j_N)\in (\nz^*)^N$.\\
 In other words any sum of N elements which individually belong to different rows in \eqref{AfN}
is in $\sigma(\Af )$. This spectrum is discrete and consists of an infinite sequence that we can order as a  non-decreasing sequence $\lambda_i$ tending to $+\infty\,$. 
We call the number of representations (which is finite)  of such a $\la$, $m(\la)$ and count the eigenvalues with multiplicity.
For  $N\ge 2$ multiplicities can indeed occur and we are mainly interested in the analysis of these multiplicities.

\begin{definition}\label{mm0}
For $N \in \mathbb N^*$, let $\mathfrak M(N)$ the family of the $\Af $'s defined by \eqref{AfN}-\eqref{ineq}. 
For each $\Af \in \mathfrak M(N)$ and $k\in \mathbb N^*$,  we introduce
\begin{equation}
m(k,\Af)= m(\lambda_k(\Af))\,,
\end{equation}
where $\lambda_k(\Af)$ is the $k$-th eigenvalue of $\Af$, and consider 
\begin{equation}\label{eq:1.7}
 \mf_k(N)=\sup_{\Af \in \mathfrak M(N)}\Big\{m(k,\Af) \Big\} \in \mathbb N^* \cup \{+\infty\}\,.
\end{equation}
\end{definition}
For fixed $N$ and $k$,  we 
want to find bounds to $\mf_k(N)$. Note that
other minimization problems relative to the harmonic oscillator (mainly in (2D)) are for example considered in \cite{Lar}.\\

The first proposition  is the following.
\begin{proposition}\label{Conj1}
For any $k$ and any $N\in \mathbb N^*$,  $\mf_k (N)$ is finite. More precisely, we have the bound:
$$
\mf_k (N) \leq k^{N-1}\,.
$$
\end{proposition}
It follows from this boundedness that we have maximizers for a given $\mf_k(N)$.

The next question is to ask for the monotonicity with respect to $k$ of $\mf_k(N)$.\\

We have the following property:
\begin{theorem}\label{thmmon}
For any $k \in \mathbb N^*$ and $N \in \mathbb N^*$, we have
\begin{equation}\label{monoton}
 \mf _k(N) \leq  \mf _{k+1}(N)\,.
\end{equation}
\end{theorem}
Note that we cannot hope to have always a strict inequality in \eqref{monoton} (see below the results for $N=3$).\\

As an intermediate result, it could be useful to validate the following conjecture:
\begin{conjecture}
For determining  $\mf_k(N)$, it suffices to take the supremum in \eqref{eq:1.7} over the  $\mathfrak A$  with the $a_i^j\in \mathbb Z$.
\end{conjecture}
This will indeed   permit to use a computer for determining $\mf_k (N)$ by looking at a finite number of possibilities. 
But the size of the computations could increase dramatically with $k$ and $N$. Actually, all the numerical results we will be referring to in this paper
are done by computation involving only integers. \\

  We do not give in this introduction a general conjecture. Our initial  guess was  that the matrix associated with the  radial harmonic oscillator was a maximizer for any $k$. Although this remark is quite 
  useful for getting a first lower bound for $\mf_k(N)$, we will show that starting from  $N=3$ this lower bound is not optimal and can be improved. \\
  The goal of this paper is to propose better lower bounds with the final hope to get or at least guess the optimal result.

 \section{Reductions}
 
For  some of the proofs it is useful to perform some normalization without changing the 
multiplicities. If for a given $\Af$, we construct a new $\widehat \Af$ by replacing 
 each element  $a_i^n\in\af^n$ by $a_i^n -a_1^n$ so that the first column in 
$\Af$ is a zero-vector:
\begin{equation}\label{firstc}
a_1^n =0\,,\, \mbox{ for } n=1,\dots, N\,,
\end{equation}
we have, since we just shift the spectrum by $-\sum_{i=1}^N a_1^i\,$, $$ m (k,\Af)=m(k,\widehat \Af)\mbox{ for any }k\in \mathbb N^*\,.$$

Without changing the multiplicities we can also interchange  the rows $\af^n$ and can also for $c>0$ consider $c\, \Af\,$. 
Hence 
we can assume after these operations that   
\begin{equation}\label{a2=1}     
a_2^1=1\le a_2^2 \leq  \dots \leq a_2^N\,.
\end{equation}
\begin{proposition}
For the determination of $\mf(k,N)$, it suffices to consider the supremum  over the family $\widehat{\mathfrak M}(N)$ of  the matrices $\Af$ of the form
\begin{equation}\label{Af0}
 \Af = \left(
\begin{matrix}
0&1 &a_3^1&\dots&a_s^1&\dots\\
0&a_2^2&a_3^2&\dots&a_s^2&\dots\\
\hdotsfor{6}\\
0&a_2^N&a_3^N&\dots&a_s^N&\dots
\end{matrix}
\right.
\end{equation}
\end{proposition}
We have under these additional assumptions  the following 
\begin{equation}\label{la1la2}
 \la_1=0\text{ and  } \la_2=1\,.
\end{equation}
We also see that $m(\la_2)\le N$ with equality if and only if $a_2^j=1$ for all $1\leq j \leq N$.

 \section{Proof of the  finiteness of $\mf_k$.}
 
For given $N$,  consider the matrix 
\begin{equation}\label{k0}
 \Af =
\begin{pmatrix}
0&a^1_2&\dots &a^1_k&\dots\\
0&a^2_2&\dots &a^2_k&\dots \\
\dots&\dots&\dots&\dots&\dots\\
0&a_2^N&\dots &a_k^N&\dots 
\end{pmatrix}
.
\end{equation}
As before we have $a_i^j<a_{i+1}^j$.  All the elements of $\Af$ are separately in $\sigma(\Af)$
since obviously $a_i^j$ belongs to the N-tuple where all the other members of this N-tuple are set zero. \\
From this, we conclude that, for any $k\geq 1$, 
\begin{equation}\label{eq:3.2}
\la_k(\Af)\le \inf_{1\le n\le N}a_k^n\,.
\end{equation}
Now if $\lambda_k(\Af) = \sum_{j=1}^N a^j_{\ell_j}$, we deduce from \eqref{eq:3.2}  that $\ell_j \leq k$, for $j=1,\dots,N$.\\
Hence the multiplicity of $\lambda_k$ is less than the cardinality of the eigenvalues for which $\ell_j \leq k$ for $j=1,\dots, N$.
This gives
\begin{equation}
m(k,  \Af) \leq k^N\,.
\end{equation}
If we now observe that once the elements of the first $(N-1)$ rows are chosen for getting $\lambda_k$, the element of the last row is determined, we get actually the better bound $k^{N-1}$, as stated in Proposition~\ref{Conj1}\,.

 \section{Monotonicity--Proof of Theorem \ref{thmmon}}

The theorem is clear for $N=1$. We now assume $N>1$.
If $\Af_k$ is a maximizer of $\mf_k(N)$ with $\lambda_{k-1}(\Af_k) < \lambda_{k}(\Af_k)$,  it is immediate, noting that $\mf_k(N)\geq 2\,$, that 
$\lambda_{k+1} (\Af_k)=\lambda_{k}  (\Af_k)$.  Hence  we get  $\mf_k(N) \leq \mf_{k+1}(N)$. \\
The argument can be generalized to the case when 
 $$
 \lambda_{\ell -1}(\Af_k) < \lambda_\ell (\Af_k) =\lambda_{\ell +1}(\Af_k) =\cdots = \lambda_{k}(\Af_k) \mbox{  with } \ell > k - \mf_k(N) +1\,.
 $$
We are in trouble if we are
 in the situation where the maximizer in $\widehat{\mathfrak M}(N)$  satisfies,  with  $\ell =k - \mf_k(N) +1$,  $$\lambda_{\ell -1}(\Af_k) < \lambda_\ell (\Af_k) =\lambda_{\ell +1}(\Af_k) =\cdots = \lambda_{k}(\Af_k)< \lambda_{k+1}(\Af_k) \,.
 $$
 Let $\lambda_k$ the $k$-th eigenvalue asssociated with $\Af_k $. We are only interested in the proof for $k\geq 2$ and we can consequently assume that $\lambda_k\geq 1\,$. 
 Let us look at all the $N$-tuples leading to $\lambda_k$. Because the multiplicity is $\geq 2\,$, we can find for each $N$-tuple   $i>1$ such that
 $$ 
 \lambda_k =\sum_{n=1}^N a_{j_n} ^n\,,
 $$
 with $a_{j_i}^i>0\,$.\\
 Between all these sums, we can also assume in addition that $a_{j_i}^i$ is maximal.\\
 We now denote by $\nu_k$ the multiplicity of the elements whose sum is $\lambda_k$ with $a_{j_i}^i$ on row  $i$.\\
 It is clear that $1\leq \nu_k \leq \mf_k\,$.\\
 We now look at the row $i$  and modify $a_{j_i}$ into $a_{j_i}-\epsilon$ ($\epsilon >0$) in such a way that any sum involving $a_{j_i}-\epsilon$ is higher that $\lambda_{k-\mf_k}$. This can easily be done by the condition 
 $$
 a_{j_i-1} < a_{j_i}-\epsilon< a_{j_i}\,.
 $$
 Then we shift $a_{j_i +1}$ into  $a_{j_i}$ and keep the other elements of this row and all elements of the other rows unchanged. \\
 Let us look at the new situation for this new matrix $\widetilde {\Af}_k$ in $\mathfrak M(N)$ .\\
 \begin{itemize}
 \item 
 The eigenvalue $\lambda_{k-\mf_k}$ is unchanged: $$\tilde \lambda_{k-\mf_k}= \lambda_{k-\mf_k}\,.$$
 \item We have an eigenvalue of multiplicity $\nu_k$ corresponding to 
 $$
  \lambda_{k-\mf_k+1} <  \tilde \lambda_{k-\mf_k+1}=\dots = \tilde \lambda_{k-\mf_k+\nu_k}< \lambda_k \,.$$
 \item We have
 an eigenvalue of multiplicity $\mf_k$ corresponding to $$
\tilde \lambda_{k-\mf_k+\nu_k +1}= \dots = \tilde \lambda_{k+\nu_k}=\lambda_k  \,.$$
 \end{itemize}
 Because $1\leq \nu_k \leq \mf_k$, we have constructed a matrix $\widetilde {\Af}_k$ in $\mathfrak M(N)$  whose $(k+1)$ eigenvalue has multiplicity $\mf_k\,$.

\section{The isotropic harmonic oscillator}
We first consider:
\begin{equation}
\Af_{N}^{har}:=\left (\begin{matrix} 0&1&2&3&\dots\\
 0&1&2&3&\dots\\
 \dots&\dots&\dots& \dots&\dots \\
  0&1&2&3&\dots
  \end{matrix}
  \right.
  \end{equation}
The spectrum of $\Af_{N}^{har}$ is indeed the spectrum of the isotropic harmonic oscillator
\begin{equation}
H_N:= \frac 12 \left(\sum_{j=1}^N (-\frac{\partial^2}{\partial x_j^2} + x_j^2 )\right)-\frac N 2\,.
\end{equation}
Hence:
\begin{equation}
\sigma(\Af_{N}^{har})= \sigma (H_N)\,.
\end{equation} 
The computation of the multiplicity for the eigenvalues of $H_N$ is well known (see for example \cite{Cha} or \cite{Ca}, Lemma 3.7.3). The spectrum is $\mathbb N$ 
 and, for $j\in \mathbb N$,  the corresponding labelling  is  $k \in \left[ \binom{N+j-1}{j-1} +1, \binom{N+j}{j} \right]$ and
  the corresponding multiplicity is $$\mu_N(j)=m(\lambda_k=j ,H_N)=\binom{N+j-1}{N-1}\,.$$
  Note that we have, for $\lambda_k=j$, ($k=k_{min}(j)$ minimal with this property) and $j\geq 1$
  $$
  k-1 = \sum_{\ell=0}^{j-1} \binom{N+\ell -1}{N-1} = \binom{N+j-1}{N}  \,.
  $$
We compare  the asymptotics of $\mu_N(j)$ and $k$.
We have for large $j$
\begin{equation}\label{muasym}
\mu_N(j)=\frac{j^{N-1}}{(N-1)!}\Big(1+\mathcal O(\frac{1}{j})\Big)
\end{equation}
and  
\begin{equation}\label{kasym}
 k -1=\frac{j^N}{N!}\Big(1+\mathcal O(\frac{1}{j})\Big).
\end{equation}
From this we have, as $j\rightarrow +\infty\,$, 
\begin{equation}
 \mu_N(j)\sim \frac{(N!)^{1-1/N} }{(N-1)!}  k_{min}(j)^{1-1/N}\,.
\end{equation}

  The first immediate lower bound  is consequently
  \begin{proposition}
   \begin{equation}
 \mf _k(N) \geq m(k,\Af_{N}^{har})\,.
 \end{equation}
\end{proposition}
  It was first  natural to  ask if 
 for any $k$ and $N$, we have equality.
This will be   indeed true for $N=2$ but we will give  a counterexample for $N=3$ and $k=4$.
Hence a new conjecture has to be found.\\
 From first computations, we propose the refined following conjecture 
 \begin{conjecture}
 \begin{equation}\label{eq:2.6}
 \mf _k(N) \geq m(k,\Af_{N}^{har})\,.
 \end{equation}
with equality in \eqref{eq:2.6}  for the $k$'s such that $\lambda_{k-1}^{har} < \lambda_k^{har}$.
 \end{conjecture}

\section{The case: N=2}
We can solve completely the problem in $2D$.

\begin{theorem}\label{N=2}
 Let $\Af\in \mathfrak M (2)$ 
and denote the associated spectrum by $\sigma(\Af)$.
Then, for $k\in \mathbb N^*$,  
\begin{equation}\label{m2}
 \mf_k(2)= \Big \lfloor\frac{1}{2}\Big(1+\sqrt{8k-7}\Big)\Big\rfloor\,.
\end{equation}
where for $x\in \mathbb R^+$, $\lfloor x \rfloor$ denotes the integer part of $x$.
\end{theorem}
{\b Proof}\\
The idea of the proof is simple. Considering $k$, $\Af$ and $\lambda$ a $k$-th eigenvalue of $\Af$, we assume that  $\la$ has multiplicity $m$. We show that this implies that there 
are at least $k-1$ points in $\sigma(\Af)$ strictly smaller than $\la$. 
By assumption, we can write $\la=\alpha_1+\beta_1=\alpha_2+\beta_2=\dots =\alpha_m+\beta_m$ where the 
$\alpha_i\in \af^1, \be_i\in \af^2$ and $\alpha_1 <\alpha_2 <    \dots <\alpha_m$. Note that this implies $\beta_m < \beta_{m-1} < \cdots < \beta_1$. So there are 
$m$ points on the row $\ell(\la)=\{(x,y)\::x+y=\la\}$. Draw the horizontal and the vertical rows through the points 
$\alpha_i,\beta_i, \;\; i=1,2,\dots m$.  For each crossing point $(\al_i,\be_j)$, $\al_i+\be_j\in \sigma(\Af)$.  
 We have to count the number of points $(\al_p,\be_q)$ such that $\al_p+\be_q<\la$ and this number  equals $k-1$.
 Then one easily show that 
\begin{equation}\label{mk}
 \frac{m(m-1)}{2}+1\le k
\end{equation}
 and this implies  the upper bound in \eqref{m2}. \\
More precisely the eigenvalues we have counted correspond to the pairs $(\alpha_1,\beta_j)$ for $j=2,\dots,m$, $(\alpha_2,\beta_\ell)$ for $\ell= 3,\dots, m$,..., $(\alpha_{m-1},\beta_m)$
whose cardinality is $ (m-1) + (m-2)+ \dots +1 = \frac{m(m-1)}{2}$.\\
The lower bound in \eqref{m2}  comes from the direct counting for the special element
\begin{equation}
\label{defAf2}
\Af_2^{har}=\left( \begin{matrix} 0 &1&2&3&\dots \\ 0 &1&2&3&\dots \end{matrix} \right.
\end{equation}
which corresponds to the spectrum of $\Af_2^{har}$.
\finbox
\\
Note that the spectrum is $\mathbb N$ and that the eigenvalue $j$ corresponds to the multiplicity $$m=(j+1)\,,$$  and that the smallest labelling $k$ of such an eigenvalue is
$$k_{min} (j) = 1 + \frac{m(m-1)}{2} = 1 + \frac{j(j+1)}{2}\,.
$$
As an exercise, we can verify that if the multiplicity is maximal for any $k$ then we are in the case of the harmonic oscillator.

\section{The case $N=3$: Main theorem}
Note that the spectrum is $\mathbb N$ and that the eigenvalue $j$ corresponds to the multiplicity 
\begin{equation}\label{multia}
 m= \binom{ 2+j}{j} = \frac{(1+j) (2+j)}{2} \,,
 \end{equation}  
and that the smallest labelling $k$ of such an eigenvalue is
$$k_{min} (j) = 1 +  \frac{j (1+j)(2+j)}{6}\,.
$$
We observe that  
\begin{equation}\label{multib} 
k_{min}(j+1)-k_{min} (j) =  (1+j) (j+2)  /2\,.
\end{equation}
 Defining the function $j(k)$
 as the unique non negative solution of  \break  $k = 1 +  \frac{j (1+j)(2+j)}{6}\,,
$ 
we obtain
 in the case of the harmonic oscillator 
\begin{equation}\label{eq:3.2a}
m(k, \Af_3^{har}) =  \frac 12 ( \lfloor j (k )\rfloor +1)( \lfloor j (k )\rfloor +2)\,.
\end{equation}
Hence the multiplicity jumps at the values $k_{min} (j)$ and the jump is given by
\begin{equation}
m(k_{min}(j+1), \Af_3^{har}) -m(k_{min}(j), \Af_3^{har}) = j+2\,.
\end{equation}

With 
$$
c(\ell) := \ell (\ell +1)/2\,,
$$
we introduce, for  $ j\in \mathbb N$, $\ell=0,\cdots,j$  and  $k \in \mathbb N^*$ in the interval 
$$ I_{j,\ell}:= [k_{min} (j) + c(j+1)-c(j+1-\ell)\,, \, k_{min} (j) + c(j+1)-c(j-\ell))\,$$
the following 
\begin{equation}\label{eq:7.5}
\overline{m} (k,  \Af_3^{har}) = m (k_{min} (j),  \Af_3^{har})  + \ell  \,.
\end{equation}
Note in particular that, for $\ell=j$, $k= k_{min} (j) + c(j+1)-c(1) = k_{min}(j+1)-1$, we have 
\begin{equation}\label{eq:7.6}
\overline{m} (k,  \Af_3^{har}) = m (k_{min}(j),  \Af_3^{har})  + j = m (k_{min}(j+1),  \Af_3^{har})-2 \,.
\end{equation}
and that, for $j\in \mathbb N$, 
$$
\overline{m} (k_{min}(j),  \Af_3^{har}) = m(k_{min}(j),  \Af_3^{har}) \,.
$$
\begin{remark}
From \eqref{eq:7.5} and \eqref{eq:7.6}, we see that
the sequence $\overline{m} (k,  \Af_3^{har})$ is a strictly increasing sequence containing all the positive integers except the sequence defined  for $j\in \mathbb N^*$ by 
$$ m (k_{min}(j+1),  \Af_3^{har})-1 =  \frac{(2+j) (3+j)}{2} -1= (j+1)(j+4)/2\,.$$
\end{remark}

The table below  permits to have a visual expression of the definition
 $$
\begin{array}{ccl}
j& k_{min}(j)& \overline{m}(k) \mbox{ for } k\in  [k_{min}(j),  k_{min}(j+1)) \\
0&1& 1\\
1& 2& 3\,\,3\,4 \\
2& 5& 6\,6\,6\,7\,7\,8 \\
3& 11&10 \,10 \,10\, 10\, 11 \, 11 \, 11\, 12\, 12 \,  13\\
4 & 21&15 \, 15\,  15 \, 15 \, 15 \, 16\, 16\, 16\, 16\, 17\, 17\, 17\, 18\, 18 \,  19\\
5 & 36 &21\, 21\,21\, 21 \, 21\, 21\, 22\, 22\, 22\, 22\, 22\, 23 \, 23 \, 23\,  23\, 24\, 24\, 24\, 25\, 25\, 26\\
6& 57 &28\,28\,28\,28\,28\,28\,28\, 29\,29\,29\,29\,29\,29\, 30\,30\,30\,30\,30\,31\,31\,31\,31\,32\,32\,32\,33\,33\,34
\end{array}
$$
  Motivated by these numerical computations we can state our main theorem, which will be proven in
Section \ref{s7}.
 \begin{theorem}\label{mainth}
 For any $k\in \mathbb N^*$, we have
 \begin{equation}\label{eq:5.7}
 \mf_k (3) \geq \overline{m} (k,  \Af_3^{har}) \,.
 \end{equation}
  \end{theorem}
 
 On the basis of our numerical computations and assuming that Conjecture~\ref{Conj1} holds,  we also conjecture that we have actually an equality  in \eqref{eq:5.7}.  We will actually prove the conjecture for $k=1,\dots, 5\,$.

\section{Case $N=3$, small values of $\mf_k (3)$.}
We can consider the maximization over the renormalized matrices
\begin{equation}\label{A3}
\Af =\left(
\begin{matrix}
0&1&a_3&a_4&\dots&a_s&\dots\\
0&b_2&b_3&b_4&\dots&b_s&\dots\\
0&c_2&c_3&c_4&\dots&c_s&\dots
\end{matrix} 
\right.
\end{equation}
with $1 \leq b_2 \leq c_2$ and the previous conditions for each row.\\
 
The main result is:
\begin{proposition}\label{m(la)1-6}
\begin{equation}
\begin{array}{l}
 \mf_1(3)=1\,,\,\mf_2(3)=3\,,  \, \mf_3(3)=3\,, \, \mf_4(3) =  4\,,  \, \mf_5 (3)= 6\,.
 \end{array}
\end{equation}
\end{proposition}
\begin{proof}
\paragraph{Computation of  $\mf_2 (3)$.}~\\
 From the analysis of the harmonic oscillator,  we already know that \break $\mf_2 (3)\geq 3 \,$. It remains to prove the upper bound.\\
 We note that $\lambda_2(\Af) = \min (1, b_2,c_2)$ and that the multiplicity is either $1$ (if $1<b_2$), $2$ if $1=b_2<c_2$, or $3$ if $1=b_2=c_2$. Hence we have established $\mf_2(3)=3$. \\

\paragraph{Computation of  $\mf_3 (3)$.}~\\
From the analysis of the harmonic oscillator, we already know that $\mf_3 (3)\geq 3 $. It remains to prove the upper bound.\\
Let us start with the case of equality in the previous discussion. That is the case when $1=b_2 $. In this case, the third eigenvalue of $\Af$  is $1$ and its multiplicity is the same (hence $\leq$) as the multiplicity of $\lambda_2$ (which is  $\leq 3$).\\
Hence, we can assume $1<b_2 \leq c_2$.\\ The second eigenvalue is simple and equal to $1$ and the third eigenvalue is
$$
\inf ( b_2, a_3)
$$
Analyzing the different cases, the multiplicity of the third eigenvalue is $1$ if $a_3 < b_2$ or  $b_2 < a_3$ and $b_2 < c_2$, $2$ is $b_2=a_3 < c_2$ or $b_2=c_2<a_3$, and $3$ if $b_2=c_2=a_3$. 
 Hence we have shown that the maximal multiplicity is $3$ but this maximal multiplicity is  also satisfied for some $\Af$ which does not correspond to the "harmonic oscillator case":
$$
\Af =\left(
\begin{matrix}
0&1&b_2&a_4&\dots&a_s&\dots\\
0&b_2&b_3&b_4&\dots&b_s&\dots\\
0&b_2&c_3&c_4&\dots&c_s&\dots
\end{matrix} 
\right.
$$
with $b_2>1$.\\
This shows that $\mf_3(3)=3$.\\

 \paragraph{Computation of  $\mf_4 (3)$.}~\\
 From the analysis of the harmonic oscillator, we already know that $\mf_4 (3)\geq 3 $. It remains to analyze the upper bound.\\ The computation of $\lambda_4(\Af) $ can be performed in the same way. 
Having in mind what was done for $\lambda_3(\Af)$ (we now can assume $\lambda_3 <\lambda_4$), it is enough to analyze two remaining cases:
\begin{itemize}
\item $1 < b_2 < a_3$ and $b_2 < c_2\,.$
\item $1=b_2 < c_2\,$.
\end{itemize}
We know indeed that for the other cases $\lambda_3=\lambda_4$ and that the multiplicity is not higher than $3$. In the first case, we have $\lambda_1=0,\lambda_2=1,\lambda_3 =b_2$
 and   $\lambda_4 =\inf (a_3,b_3,c_2)$  which leads to a maximal multiplicity $3$.\\
 In the second case, we have $\lambda_1=0\,,\, \lambda_2=\lambda_3 =1$ and $\lambda_4= \inf (a_3,b_3,c_2,2)$ which leads to a maximal multiplicity $4$ which is obtained when $a_3=b_3=c_2=2\,$. \\
 Hence we have shown that $\mf_4(3)=4$ and this disproves  the conjecture that the maximal multiplicity 
is obtained by the isotropic harmonic oscillator.\\

 \paragraph{Computation of  $\mf_5 (3)$.}~\\
The computation of $\lambda_5(\Af) $ is  more lengthy.  But we have no more to consider the cases where $\lambda_4 (\Af)=\lambda_5 (\Af)$, where the multiplicity is less than $4\,$.\\

Looking at the computation of $\lambda_4$ we have to consider the case when $b_2=1$, and  $\inf (a_3,b_3,c_2,2)$ is attained for one value. \\
Hence, we are led to analyze four subcases:
\begin{itemize}
\item $a_2 =b_2=1$, \, $a_3 < b_3$,\, $a_3 < c_2$, \,$a_3 <2$\,.
\item $a_2 =b_2=1$, $b_3 < a_3$, $b_3 < c_2 $, $b_3 < 2$\,.
\item $a_2 =b_2=1$, $c_2 < a_3$, $c_2 < b_3$, $c_2 < 2$\,.
\item $a_2 =b_2=1$, $2 < a_3$, $2 < b_3$, $2 < c_2$\,.
\end{itemize}
For  the first subcase, we have $\lambda_1=0$, $\lambda_2=\lambda_3=1$,  $\lambda_4=a_3$, $$\lambda_5 =\inf(b_3,c_2,2, a_4)\,.$$  
Hence the multiplicity is not higher than $4$ with equality when $$2=b_3=c_2= a_4\,.$$
For  the second subcase, we have $\lambda_1=0$, $\lambda_2=\lambda_3=1$,  $\lambda_4= b_3$, $$\lambda_5 =\inf(a_3,b_4, c_2,2)\,.$$  
Hence the multiplicity is not higher than $4$ with equality when $2=a_3=b_4=c_2$.\\
For the third subcase,  we have $$\lambda_1=0\,, \,\lambda_2=\lambda_3=1\,,  \, \lambda_4= c_2\,,$$ and  $$\lambda_5 =\inf(a_3,b_3, c_3,2, 1+c_2)\,,$$  
where $(1+c_2)$ has to be counted with multiplicity $2$. Hence the multiplicity is not higher than $6$ with equality when $$2=a_3=b_3=c_3 =1+c_2\,,$$ which corresponds to the isotropic harmonic case.\\
For  the last subcase, we have $\lambda_1=0\,$, $\lambda_2=\lambda_3=1\,$,  $\lambda_4= 2\,$, $$\lambda_5 =\inf(a_3,b_3, c_2)\,.$$  
Hence the multiplicity is not higher than $3$ .\\

We now consider in reference to what was done for $\lambda_4$ the case when
 $1 < b_2 < a_3$ and $b_2 < c_2$. Here we recall that  $\lambda_1=0,\lambda_2=1,\lambda_3 =b_2$
 and $\lambda_4 =\inf (a_3,b_3,c_2)$. Hence, in order to have $\lambda_4 < \lambda_5$, we have to consider three subcases
 \begin{itemize}
 \item $a_3 < b_3$ and $a_3 < c_2$\,,
 \item $b_3 < a_3$ and $b_3 < c_2$\,,
 \item $c_2 < a_3$ and $c_2 < b_3$\,.
 \end{itemize}
 In the first subcase, we have $\lambda_4= a_3$ and $\lambda_5= \inf (b_3,c_2,a_4, 1+b_2)$. Hence the multiplicity is not higher than $4$ .\\
 In the second subcase, we have $\lambda_4= b_3$ and $\lambda_5= \inf (a_3,c_2, b_4, 1+b_2)$. 
 Hence the multiplicity is not higher than $4$ .\\
 In the last subcase, we have $\lambda_4= c_2$ and $\lambda_5= \inf (b_3,c_3,a_3, 1+b_2, 1+c_2)$. 
 Hence the multiplicity is not higher than $5\,$.\\
 
  Hence we have shown that $\mf_5(3)=6\,$.
 \end{proof}
 \begin{remark}
In order to obtain $\mathfrak m_k(3)$ for $k>5$, say,
for instance $k=8$ where we expect $\mathfrak m_8=7$  we would run into very involved combinatorics
and the number of cases to consider grows dramatically with $k\,$.
\end{remark}

 \section{Proof of the main theorem for $N=3$}\label{s7}
 We now start the proof of the main theorem. We observe that for $k\in [k_{min}(j), k_{min}(j) +j+1 )=I_{j,0}$ the claim is nothing else as the claim obtained from the radial harmonic oscillator. The proof will be actually given by starting from $k= k_{min}(j+1)$ and considering decreasing $k$'s.
  \subsection{Proof for the first jump}
Taking $k= k_{min}(j+1)$, with $j>0$, we observe that $$\lambda_{k-1}^{har} < \lambda_k^{har} = \lambda_{k-1}^{har} +1 =j+1   $$
and will prove that
  \begin{equation}\label{eq:2.7}
 \mf _{k-1}(3) \geq m(k,\Af_{3}^{har})-2\,.
 \end{equation}
 We recall that $ m(k,\Af_{3}^{har})-2= \overline {m}(k-1,\Af_{3}^{har})$ for $k = k_{min}(j+1)$.\\
For this, we perturb the maximizer $\Af_3^{har}$ into $\widetilde{\Af}_3^{har}$ by replacing in the third row of  $\Af_3^{har}$ $\lambda_{k-1}^{har}$ by $\lambda_{k-1}^{har} +1$, $\lambda_k^{har}$  by $\lambda_{k}^{har} +1$ and so on (but the other terms will not play a role in the argument). In this way the multiplicity of $\lambda_{k-1}^{har}$ 
 in $\sigma(\widetilde{\Af}_3^{har})$ decreases by one (we loose $0+0+\lambda_{k-1}^{har}$)  and hence the lowest labelling of $\lambda_{k}^{har}$  in $\sigma(\widetilde{\Af}_3^{har})$ becomes $k-1$. 
 Analyzing the possible sums leading to $\lambda^{har}_k$, two sums disappear: $1+0 +\lambda_{k-1}^{har}$ and $0+1 +\lambda_{k-1}^{har}$,   and hence the new multiplicity of $\lambda_{k}^{har}$ in   $\sigma(\widetilde{\Af}_3^{har})$ decreases by $2$.  So we get
 $$
\mf_{k-1}(3)\geq  m (k-1,\widetilde{\Af}_3^{har})= m(k, \Af_3^{har})-2\,,
 $$ 
 hence \eqref{eq:2.7} as announced.

  \subsection{Proof for the second jump}
  We were inspired by the maximizers appearing in some non exhaustive numerics by perturbation on one or two rows of $\Af_k^{har}$. \\ We assume that $k\geq 4$. Hence the multiplicity of $\lambda_{k-1}$ is larger than $3$.
 This time we  perturb $\Af_3^{har}$ into $\widehat{\Af}_3^{har}$ by replacing in the third row of  $\Af_3^{har}$ $\lambda_{k-1}^{har}-1=\lambda_k^{har}-2$ by $\lambda_{k-1}^{har} $  and $  \lambda_{k-1}^{har}$ by  $  \lambda_{k-1}^{har} +1  $ and so on. In other words, we delete in the last row the $(k-2)$-th term and then shift the next terms.\\
 In this way, the multiplicity of  $\lambda_{k-1}^{har}-1$ in $\sigma(\widehat{\Af}_3^{har})$ decreases by one, the multiplicity of $\lambda_{k-1}^{har}$ 
 in $\sigma(\widehat{\Af}_3^{har})$ decreases by two (we loose $(1+0+(\lambda_{k-1}^{har}-1))$ and $(0+1+(\lambda_{k-1}^{har}-1))$  and hence the lowest labelling  of $\lambda_{k}^{har}$  in $\sigma(\widehat{\Af}_3^{har})$ becomes $k-3$. 
  Analyzing the possible sums leading to $\lambda^{har}_k$ for the spectrum of $\widetilde{\Af}_3^{har}$, three sums disappear: $0+2 +(\lambda_{k-1}^{har}-1)$,  $1+1 +(\lambda_{k-1}^{har}-1)$ and
  $2+0 +(\lambda_{k-1}^{har}-1)$,   and hence the new multiplicity of $\lambda_{k}^{har}$ in   $\sigma(\widehat{\Af}_3^{har})$ decreases by $3$. So we get
 \begin{equation}\label{eq:9.2}
\mf_{k-3}(3)\geq  m (k-3,\widehat{\Af}_3^{har})= m(k, \Af_3^{har})-3\,.
 \end{equation}
  Now, if $m(k, \Af_3^{har})-3 \geq 2$, we get $ m (k-3,\widehat{\Af}_3^{har})\geq 2$ which implies, because $k-3$ was the minimal labelling of $\lambda_k^{har}$
  $$
  \mf_{k-2}(3)\geq  m(k, \Af_3^{har})-3\,.
 $$
 One can also directly deduce this last inequality from the monotonicity of the multiplicity with respect to $k$.\\
  Hence we have shown the theorem for $k$ in the interval $$ [k_{min}(j+1) -3, k_{min}(j+1) - 1)= I_{j,j-1}\,.$$
This is coherent with our numerics showing that  the same matrix  can be used for the lower bound of  $\mf_{k-2} (3)$ and  $\mf_{k-3} (3)$.
\subsection{Proof for any jump}
We can continue in the same way assuming that $\lambda_{k-3} =\lambda_{k-2} =\lambda_{k-1}$ (assuming that $3 \leq  m(k-1,\Af_3^{har})$)  and deleting this time  $\lambda_k -3$ on the last row
 and shifting the next terms. The labelling  of $\lambda_k$ decreases by $1 +2 + 3$ and its multiplicity will decrease by $4$. This gives
 $$
 \mf_{k-6} \geq  m(k, \Af_3^{har}) - 4\,.
 $$
 This gives the scheme for the general proof of the theorem. The condition on the multiplicity of $\lambda_{k_{min}(j)}=j$ which appears in the proof
  is indeed satisfied as soon as $j\geq 2\,$. Using the monotonicity, we also obtain
   $$
 \mf_{k-4} \geq \mf_{k-5} \geq \mf_{k-6} \geq  m(k, \Af_3^{har}) - 4\,.
 $$
  Coming back to the definition of $\overline{m} (k, \Af_3^{har})$, this proves the theorem for  $k\in [k_{min}(j+1) -6, k_{min}(j+1) - 3)= I_{j,j-2}\,$.
  \section{About higher dimensions}\label{s10}
  \subsection{$  \mf_2(N) $}
  It is evident that $\mf_1(N)=1$ and looking at the proof of the case $N=3$, we obtain
  \begin{equation}
  \mf_2(N) =N\,.
  \end{equation}
 Let us detail the case $N=4$. From the analysis of the harmonic oscillator,  we already know that $\mf_2 (4)\geq 4 $. It remains to prove the upper bound.\\
 We note that $\lambda_2(\Af) = \min (1, b_2,c_2,d_2)$ and that the multiplicity is either $1$ (if $1<b_2$), $2$ if $1=b_2<c_2$, $3$ if $1=b_2=c_2<d_2$ or $4$ if $1=b_2=c_2=d_2$. Hence we have established $\mf_2(4)=4$. 
 \subsection{Lower bounds continued}\label{ss10.2}
 Because we played in the case $N=3$  with only the last row with the hope to be optimal according to various numerical tries, we can get  lower bounds  in the same way  for any $N$, but we will show that this cannot be optimal already in the case $N=4$.\\
   Let us for example  follow the argument for the first jump. Starting from a $k$ corresponding to a labelling such that  $\lambda_{k-1} < \lambda_k$. Then the same argument gives
  \begin{equation}\label{eq:10.1}
\mf_{k-1}(N)\geq   m(k, \Af_N^{har})-N +1\,.
 \end{equation}
 For the second jump, we get, under the condition that $N$ is smaller than the multiplicity of the eigenvalue of the harmonic oscillator $m(k-1, \Af_N^{har})$ (true if $k>2$), 
 \begin{equation}
\mf_{k-N}(N)\geq   m(k, \Af_N^{har})- \mu(2,N-1)\,,
 \end{equation}
 where $\mu(2,N-1)$ is the multiplicity of the eigenvalue $2$ for $ \Af_{N-1} ^{har}$, i.e. 
 $$
 \mu(2,N-1)= N (N-1)/2\,.
 $$
 For $N=4$, this reads, having in mind for $\Af_4^{har}$ that
 $$ \la_1=0, 1= \la_2=\la_3=\la_4= \la_5 , 2 =\la_6= ...= \la_{15}\,,$$
and hence $$m(\la_1,\Af_4^{har})=1, m(\la_2,\Af^{har}_4)=4, m(\la_6,\Af^{har}_4)=10\,,\, m(\lambda_{16},\Af^{har}_4)=20\,.
$$
Hence at this stage, we have 
$$\begin{array}{l}
\mf_1(4)=1, \mf_2(4) = 4\,, \mf_3(4) = 4,  \mf_4(4) \geq 4, \mf _5(4)\geq 7, \mf_6(4) \geq 10\,,\\
\mf_{12} (4) \geq 14, \mf_{15} (4) \geq 17, \mf_{16}(4)\geq 20\,.
\end{array}
$$
These estimates can be combined with the monotonicity argument.

\section{The case $N=4$, $k=1,\dots,5$}
 \subsection{Computation of $\mf_3(4)$}
 Let us look at $\mf_3(4)$. We use the notation
 $$
\Af =\left(
\begin{matrix}
0&1&a_3&a_4&\dots&a_s&\dots\\
0&b_2&b_3&b_4&\dots&b_s&\dots\\
0&c_2&c_3&c_4&\dots&c_s&\dots\\
0&d_2&d_3&d_4&\dots&c_s&\dots
\end{matrix} 
\right.
$$
and recall that $1 \leq b_2\leq c_2\leq d_2\,$.\\
As observed in the case $N=3$, we have already shown that the multiplicity of $\lambda_3(\Af)$ is $\leq 4$ unless $1 < b_2$.\\
 Hence from now on, we assume that
 $$
 1 < b_2\leq c_2 \leq d_2\,.
 $$
 The second eigenvalue is simple and equal to $1$ and the third eigenvalue is
$$
\inf ( b_2, a_3)\,.
$$
There are three cases to consider $b_2>a_3$, $b_2=a_3$ or $b_2<a_3$.\\
In the first case, the multiplicity is $1$.\\
In the second case, the multiplicity is $4$ if $a_3=b_2=c_2=d_2$, $3$ if $a_3=b_2=c_2<d_2$, $2$ if $a_3=b_2 < c_2$.\\
Finally, we have to look at the case $b_2<a_3$.  The multiplicity is $3$ if $b_2=c_2=d_2$, $2$ if $b_2=c_2<d_2$, $1$ if $b_2 < c_2$.\\
Hence, we have also proved that $\mf_3(4) = 4$.

\begin{remark}\label{rem11.1}
In the second case just above, we note that the multiplicity of $\lambda_4(\Af)$ is $4$ which gives already  the lower bound for $\mf_4(4)\geq 4$.
\end{remark}

\subsection{Computation of $\mf_4(4)$}\label{ss11.2}
In our analysis, we have not to come back to the previously analyzed following cases  where for some $\Af$
\begin{itemize}
\item $m(2,\Af )\ge 3\,$,
\item $\la_2(\Af)<\la_3(\Af) $ and  $m(3,\Af)\ge 2\,$.
\end{itemize}
We have already a lower bound for these cases  (see Remark \ref{rem11.1}). Hence  if we  find a new example with a higher multiplicity than $\mf_3(4)$, we will also get a new lower bound. So we can assume that either
 $$
(Case\, A): \quad \lambda_2(\Af)=\lambda_3(\Af) < \lambda_4(\Af)\,,
 $$
 or
 $$
(Case \,B): \quad  \lambda_2(\Af) < \lambda_3(\Af) < \lambda_4(\Af)\,.
 $$
 {\bf In Case A}, we have $1=b_2<c_2$ and $\lambda_2(\Af)=\lambda_3(\Af)=1$. The fourth eigenvalue is
 $$
 \lambda_4(\Af) = \inf (a_3, b_3, c_2,d_2, 2)\,.
 $$
 This gives an upper bound of the multiplicity by $5$, where $5$ is obtained for
 $$
 a_3=b_3=c_2=d_2=2\,,
 $$
 corresponding to $\lambda_4(\Af)=2$  and $m(4,\Af)=5$.\\
 \begin{remark}\label{rem}
 We note that in this case of equality  $\lambda_3(\Af) < \lambda_4(\Af)$ and not only  $\mf_4(4) \geq 5$ but also $\mf_5(4) \geq 5$ .
 \end{remark}

 This matrix reads
  $$
\Af =\left(
\begin{matrix}
0&1&2&3&\dots&\dots\\
0&1&2&3&\dots&\dots\\
0&2&3&4&\dots&\dots\\
0&2&3&4&\dots& \dots
\end{matrix} 
\right.
$$
 We note that we have modified two rows starting from the harmonic case.
{\bf  Let us now look at Case B.\\}
 So we have $1 < b_2$.\\
 According to the analysis of $\mf_3(4)$, we can distinguish two cases corresponding to $\lambda_3$ simple.
 $$
 (Case \,B1):\quad  1 < b_2 \mbox{ and } a_3 < b_2
 $$
 or
 $$
  (Case \,B2):\quad  1 < b_2, b_2 < c_2  \mbox{ and } b_2 < a_3\,.
 $$
 In Case B1, we have $\lambda_3(\Af) = a_3$ and
 $$
 \lambda_4(\Af)= \inf ( a_4, b_2,c_2,d_2)\,,
 $$
 so the maximal multiplicity is four with in this case $a_4=b_2=c_2=d_2 >a_3$.\\
 In Case B2, we have $\lambda_3(\Af) = b_2$ and
 $$
 \lambda_4(\Af)= \inf (a_3,b_3,c_2,d_2,1+b_2)
 $$
 so the maximal multiplicity is $5$ and obtained for
 $$
 a_3=b_3=c_2=d_2=1+b_2\,.
 $$
 Here we get a continuous family of maximizers, which contains  nevertheless a representative with integer coefficients. Note that for this model we have perturbed all the rows.
  In any case, we have proven that $\mf_4(4)=5$ with two different classes of maximizers.  Moreover, we have shown that the conjecture that the lower bound given for $N=3$ in Theorem \ref{mainth}  could be 
   an upper bound cannot be extended to the case $N=4\,$.\\
 
 \subsection{Computation of $\mf_5(4)$}\label{ss11.3}
 In our analysis,  as observed at the beginning of Subsection \ref{ss11.2}, we have not to come back to the cases when in the analysis of $\lambda_2(\Af)$ of some $\Af$ we have shown that the multiplicity was $\geq 4$, for $\lambda_3(\Af)$
 when we show that the multiplicity was $\geq 3$ (with $\lambda_2(\Af) < \lambda_3(\Af)$) and for  $\lambda_4(\Af)$
 when we show that the multiplicity was $\geq 2$ (with $\lambda_3(\Af) < \lambda_4(\Af)$).
 
 Hence it remains to consider:
 \begin{itemize}
 \item[(A5)] $\lambda_2(\Af)=\lambda_3(\Af)=\lambda_4(\Af) < \lambda_5(\Af)$\,,
 \item[(B5)]  $\lambda_2(\Af)< \lambda_3(\Af)=\lambda_4(\Af) < \lambda_5(\Af)$\,,
 \item[(C5)] $ \lambda_3(\Af) < \lambda_4(\Af) < \lambda_5(\Af)$\,.
 \end{itemize}
 
 In case (A5), we have $1=b_2=c_2<d_2$ and $\lambda_4(\Af)=\,1$.\\
 We immediately get that
 $$
 \lambda_4(\Af)< \lambda_5(\Af) = \inf (  d_2, a_3, b_3,c_3,2 )\,.
 $$
 The maximal multiplicity is when $d_2=a_3=b_3=c_3=2$ and the multiplicity corresponding to this $\Af$  is $m(5,\Af)=7$. This corresponds to just deleting $1$ in the last row.\\~\\
 In case (B5), we have $1< b_2 \leq   c_2$. Here according to the analysis of $\mf_3(4)$, we have two  subcases to consider in order to get $\lambda_3(\Af)$ of multiplicity 2: $$a_3=b_2 < c_2 \mbox{  or } b_2=c_2<\inf (a_3,d_2)\,.$$
In the two subcases, we have  $\lambda_3(\Af)=\lambda_4(\Af)= b_2>1$ and, in Subcase (B5a)
 $$
 \lambda_5(\Af)= \inf (c_2,d_2,a_4,b_3, 1+b_2 )\,,
 $$
 while in Subcase (B5b) we have
  $$
 \lambda_5(\Af)= \inf (d_2,a_3,b_3, 1+b_2 )\,.
 $$
 In Subcase (B5a), the maximal multiplicity is $5$ and in Subcase (B5b), the maximal multiplicity is also $5\,$.\\
 We now treat the case (C5). Here we have 
 \begin{itemize}
 \item[(C5a)] $ \lambda_2(\Af)=\lambda_3(\Af) < \lambda_4(\Af) < \lambda_5(\Af)$\,,
 \item[] or 
 \item[(C5b)] $ \lambda_2(\Af)< \lambda_3(\Af) < \lambda_4(\Af) < \lambda_5(\Af)$\,.
 \end{itemize}
 
 In case (C5a), we have $1=b_2<c_2$. The fourth eigenvalue is
 $$
 \lambda_4(\Af) = \inf (a_3, b_3, c_2,d_2, 2)\,.
 $$
 In order to get $ \lambda_4(\Af)$ simple, the infimum should be attained for only one element between the five elements appearing in the infimum.
 \begin{enumerate}
 \item If $a_3 < \inf (b_3, c_2,d_2, 2)$, we have $\lambda_5(\Af)= \inf (a_4, b_3, c_2,d_2, 2)$ and the multiplicity is at most $5\,$.
 \item If $b_3 < \inf (a_3,  c_2,d_2, 2)$, we have $\lambda_5(\Af)= \inf (a_3, b_4, c_2,d_2, 2)$ and the multiplicity is at most $5\,$.
\item if $c_2< \inf (a_3, b_3,d_2, 2)$, we have $\lambda_5(\Af)= \inf (a_3, b_4, c_3,d_2, 2)$ and the multiplicity is at most $5\,$.
\item If $d_2 <  \inf (a_3, b_3, c_2, 2)$, we have $\lambda_5(\Af)= \inf (a_3, b_3, c_2,d_3, 2)$ and the multiplicity is at most $5\,$.
\item If $2 <  \inf (a_3, b_3, c_2, d_2)$, we have $\lambda_5(\Af)= \inf (a_3, b_3, c_2,d_3)$ and the multiplicity is a most $4\,$.
 \end{enumerate}
 
 In Case (C5b), we have $1 < b_2 $ and come back to the discussion of Case B for $\mf_4(4)$.\\
 $$
 (Case \,B1):\quad  1 < b_2 \mbox{ and } a_3 < b_2\,.
 $$
 We have $\lambda_3(\Af) = a_3$ and
 $$
 \lambda_4(\Af)= \inf ( a_4, b_2)\,.
 $$
 We have to analyze two subcases:
 \begin{itemize}
 \item If $a_4 <b_2$, we have $\lambda_5(\Af)= \inf ( a_5,b_2,c_2,d_2)$ and  the multiplicity is at most $4\,$.
 \item If $b_2 <\inf ( a_4, c_2)$, we have $\lambda_5(\Af)= \inf ( a_4, c_2, b_3, 1+b_2)$ and  the multiplicity is at most $5\,$.
 \end{itemize}
 Our last subcase to consider is
 $$
  (Case \,B2):\quad  1 < b_2, b_2 < c_2  \mbox{ and } b_2 < a_3\,,
 $$
 where we have  $\lambda_3(\Af) = b_2$ and
 $$
 \lambda_4(\Af)= \inf (a_3,b_3,c_2,d_2,1+b_2)=\inf (a_3,b_3,c_2,1+b_2)\,.
 $$
 To have $\lambda_4(\Af)$ of multiplicity $1$, we get four subcases
 \begin{itemize}
 \item If $a_3 < \inf ( b_3,c_2,1+b_2)$, we  have $\lambda_5(\Af)=  \inf (a_4, b_3,c_2,d_2, 1+b_2)$ and  the multiplicity is at most $5\,$.
 \item If $b_3 < \inf (c_2,1+b_2)$, we  have $\lambda_5(\Af)= \inf (a_3,b_4,c_2,d_2,1+b_2)$  and  the multiplicity is at most $5\,$.
 \item If $c_2 < \inf (a_3,b_3,d_2,1+b_2)$, we  have $\lambda_5(\Af)= \inf (a_3,b_3,c_3,d_2,1+b_2)$  and  the multiplicity is at most $5\,$.
 \item If $(1+b_2) < \inf (a_3,b_3,c_2)$,  we  have $\lambda_5(\Af)= \inf (a_3,b_3,c_3,c_2)$  and  the multiplicity is at most $5\,$.
 \end{itemize}
 One remarks that only (A5) gives a $\Af$ (obtained by deleting $1$ in the last row) for which $m(4,\Af)=7$, that  all the other cases have multiplicity $\le 5$
and that multiplicity 6 never occurs.
Finally, this gives $\mf_5(4)=7$.
 \subsection{Conclusion for $N=4$}
 We have proven
 $$\begin{array}{l}
\mf_1(4)=1, \mf_2(4) = 4\,, \mf_3(4) =4,  \mf_4(4)=5, \mf _5(4)= 7, \mf_6(4) \geq 10\,,\\
\mf_7(4)\geq 10,   \mf_8(4)\geq 10, \mf_9(4)\geq 10, \mf_{11}(4)\geq 10, \\
       \mf_{12} (4) \geq 14,    \mf_{13} (4) \geq 14,  \mf_{14} (4) \geq 14, \mf_{15} (4) \geq 17, \mf_{16}(4)\geq 20\,.
\end{array}
$$
On the other hand, numerics give the following lower bounds
$$\begin{array}{l}
 \mf_6(4) \geq 10, 
\mf_7(4)\geq 10,   \mf_8(4)\geq 10, \mf_9(4)\geq 10,  \mf_{11}(4)\geq 12, \\
       \mf_{12} (4) \geq 14,    \mf_{13} (4) \geq 14,  \mf_{14} (4) \geq 14, \mf_{15} (4) \geq 17,\\
        \mf_{16}(4)\geq 20, 
       \mf_{17} (4) \geq 20,  \mf_{18} (4) \geq 20, \mf_{19} (4) \geq 20, \dots, \mf_{22} (4) \geq 20, \\ \mf_{23}(4) \geq 21.
\end{array}
$$
The new case in blue corresponds to the following matrix 
 $$
\Af =\left(
\begin{matrix}
0&1&2&3&\dots&\dots\\
0&1&2&3&\dots&\dots\\
0&2&3&4&\dots&\dots\\
0&1&3&4&\dots& \dots
\end{matrix} 
\right.
$$
 The table below  permits to have a visual expression of the lower bounds $\check m(k)$  given by numerics  or the theory
 $$
\begin{array}{ccl}
j& k_{min}(j)& \check{m}(k) \mbox{ for } k\in  [k_{min}(j),  k_{min}(j+1)) \\
0&1& 1\\
1& 2& 4\,\,4\,{\clb 5}\, 7 \\
2& 6 & 10\,10\,10\,10\,10\,{\clb 12}\, 14\,14\,14\, 17  \\
3& 16 &20 \,20 \,20\, 20\, 20 \, 20 \, 20\, {\clb 21\,21\, 23}\,25\,25\,25\,25\,25\,{\clg 26}\,29\,29\,29\, 32\\
4& 36& 35\,35\,35\,35\,35\,35\,35\,35\,35\,35\,35\,35\,{\clb 37\,37\,39}\,41\,41\,41\,41\,41\,41\,41\,41\,41\,\dots
\end{array}
$$
The multiplicities in blue correspond to matrices with two rows modified starting from $\Af_4^{har}$.\\
 The multiplicity  $26$ in green corresponds to the $31$-th eigenvalue (which equals $4$), obtained  by deleting $2$ and $3$ on the last row of $\Af_4^{har}$:\\
 $$
\Af_4^{har \setminus \{2,3\} )}=\left(
\begin{matrix}
0&1&2&3&\dots&\dots\\
0&1&2&3&\dots&\dots\\
0&1&2&3&\dots&\dots\\
0&1&4&5& \dots& \dots
\end{matrix} 
\right.
$$
Note that the matrix, with two rows modified starting from $\Af_4^{har}$,
  $$
\Af =\left(
\begin{matrix}
0&1&2&3&\dots&\dots\\
0&1&2&3&\dots&\dots\\
0&1&3&4&\dots&\dots\\
0&1&2&4& \dots& \dots
\end{matrix} 
\right.
$$
gives the same multiplicity for the $31$-th eigenvalue.\\ 
 Contrarily to the case $N=3\,$, we do not see an obvious plausible structure for this table and cannot present a reasonable conjecture in the case $N=4\,$. 
This table is to compare with the lower bounds which can be  obtained as for $N=3$  by the arguments of Subsection \ref{ss10.2}, only deleting one integer in the last row of the harmonic example.

 $$
\begin{array}{ccl}
j& k_{min}(j)& \overline{m}(k) \mbox{ for } k\in  [k_{min}(j),  k_{min}(j+1)) \\
0&1& 1\\
1& 2& 4\,\,4\, 4\, 7 \\
2& 6 & 10\,10\,10\,10\,10\,{\clb 10 }\, 14\,14\,14\, 17  \\
3& 16 &20 \,20 \,20\, 20\, 20 \, 20 \, 20\, {\clb 20\,20\, 20}\,25\,25\,25\,25\,25\,{\clb 25}\,29\,29\,29\, 32\\
4& 36& 35\,35\,35\,35\,35\,35\,35\,35\,35\,35\,35\,35\,{\clb 35\,35\,35}\,41\,41\,41\,41\,41\,41\,41\,41\,41\,\dots
\end{array}
$$

\section{New approach (for the same statements)}
\subsection{Decomposition formula}
We now give a computation which could be general for computing the multiplicities and the corresponding energies when only one row is modified.
The first point is to see how we can recover the multiplicities for $\Af_4^{har}$ by expansion along the last row. The formula reads
$$
\mu (j,\Af_4^{har}) = \sum_{\ell=0}^j \mu (j-\ell,\Af_3^{har})\,.
$$
Here $\mu(\lambda,\Af)$  is the multiplicity of the eigenvalue $\lambda$ of $\Af\,$. 
\subsection{Coming back to the last example}
Using the same expansion for $\Af_4^{har \setminus \{2,3\} )}$, we get
\begin{equation}\label{eq:31a}
\mu(j,\Af_4^{har \setminus \{2,3\} )}) = \sum_{\ell \in \{0,\dots,j\} \setminus\{2,3\} }\mu (j-\ell,\Af_3^{har})\,.
\end{equation}
To recover the minimal labelling corresponding to energy $j$, we simply write for $j\geq 1$
\begin{equation}\label{eq:31b}
k_{min} (j) = 1+ \sum_{n=0}^{j-1}   \mu (n, \Af_4^{har \setminus \{2,3\} })\,.
\end{equation}

\paragraph{Explicit computations}~\\
We recall the formulas obtained  for $\Af_3^{har}$ with the corresponding multiplicity $1,3,6, 10, 15$ of the eigenvalues $0,1,2,3,4$ and the minimal labellings $1,2,5, 11,21$.
Applying \eqref{eq:31a}, for $j=0,\dots, 4$, we get  for $\mu(j,\Af_4^{har \setminus \{2,3\} )})$ the sequence
$$1\,,\, 3+1=4\,,\, 6 + 3=9\,,\, 10+6=16\,,\,15+10 +1=26    \,.$$
According to \eqref{eq:31b}, the corresponding minimal labellings are $1,2, 6, 15, 31$. This confirms what we claim above: there exists a matrix $\Af$, whose $(31$)-th 
eigenvalue has multiplicity $26$.\\

\subsection{Coming back to the case $N=4$.}
We note that this approach is quite general and permits easy computations for all the cases considered previously. Let us for example, see what we get, when deleting 
$k-1$ on the last row and looking at energy $j=k$.\\ Hence we consider, the matrix $\Af_4^{har \setminus \{k-1\}}$.\\
It is clear that the minimal labellings  are unchanged for the energies $\leq  k-1$. For the energy $j=k$, the minimal labelling 
 of $j=k$ is given by
 $$
 \begin{array}{ll}
 k_{min} (k) & = 1+ \sum_{n=0}^{k-1}   \mu(n, \Af_4^{har \setminus \{k-1\} })\\
 &  =  1+ \sum_{n=0}^{k-2}   \mu (n, \Af_4^{har}) + \mu(k-1, \Af_4^{har \setminus \{k-1\} })  \\
 & = k_{min}^{har} (k)  -  \mu  (0,\Af_3^{har}   )   \,.
 \end{array}
 $$
 By comparison of the multiplicity formulas, we get
$$
\mu (k, \Af_4^{har \setminus \{k-1\}} )= \mu (k, \Af_4^{har}) + \mu (1,\Af_3^{har})\,.
$$
We recover the first result of Section~\ref{s10} and could give the general description of what  can be obtained by playing on the last row.
\subsection{Coming back to $N=3$}
Applied in the case $N=3$, this idea could  clarify the proof given in Section~\ref{s7}. We can indeed start from the formula
\begin{equation}
\mu (j,\Af_3^{har}) =\sum_{\ell=0}^j \mu (j-\ell, \Af_2^{har})\,.
\end{equation}
For the first jump, if we delete  the eigenvalue $j-1$ (and shift)  in the last row (this corresponds to our choice of $\widetilde{\Af}_3^{har}=\Af_3^{har\setminus\{j-1\}}$, we get  for energy $j$, by comparison
\begin{equation}
\mu (j,\Af_3^{har}) = \mu (j, \widetilde{\Af}_3^{har}) + \mu (1,\Af_2^{har})=  \mu (j, \widetilde{\Af}_3^{har}) +2  \,.
\end{equation}
The labelling of $j$ for $\Af_3^{har}$ is $k_{min} (j)$. It remains to give the corresponding labelling $\tilde k_{min} (j)$ for $\widetilde{\Af}_3^{har}$.
Here we use the formula
\begin{equation}\label{eq:12.5}
k_{min} (j) = 1+ \sum_{n=0}^{j-1}\mu (n,\Af_3^{har})\,.
\end{equation}
Similarly we have 
\begin{equation}\label{eq:12.6}
\tilde k_{min} (j) = 1+ \sum_{n=0}^{j-1}\mu (n,\widetilde{\Af_3}^{har})\,.
\end{equation}
By comparison of the two formulas, we have, for $j\geq 1$, 
$$
k_{min} (j) = \tilde k_{min} (j) + \mu (j-1,\Af_3^{har})- \mu (j-1,\widetilde{\Af_3}^{har})= \tilde k_{min} (j) +\mu(0, \Af_2^{har})   \,.
$$
hence 
\begin{equation}\label{eq:12.7}
k_{min} (j)   =\tilde k_{min} (j)  +1\,.
\end{equation}

Hence we get back \eqref{eq:2.7}.\\
For the second jump, if we delete  the eigenvalue $j-2$ (and shift)  in the last row (this corresponds to our choice of $\widehat {\Af}_3^{har} = \Af_3^{har\setminus\{j-2\}}$), we get by comparison
\begin{equation}\label{eq:12.8}
\mu (j,\Af_3^{har}) = \mu (j, \widehat{\Af}_3^{har}) + \mu (2,\Af_2^{har})=  \mu (j, \widehat{\Af}_3^{har}) +3  \,.
\end{equation}
The labelling of $j$ for $\Af_3^{har}$ is $k_{min} (j)$. It remains to give the corresponding labelling $\hat k_{min} (j)$ for $\widehat{\Af}_3^{har}$.
Here we use  the formulas \eqref{eq:12.5} and 
\begin{equation}\label{eq:12.9}
\hat k_{min} (j) = 1+ \sum_{n=0}^{j-1}\mu (n,\widehat{\Af_3}^{har})\,.
\end{equation}
By comparison of the two formulas, we have, for $j\geq 1$, 
$$
\begin{array}{ll}
k_{min} (j) & = \hat k_{min} (j) + \mu (j-1,\Af_3^{har})- \mu (j-1,\widehat{\Af_3}^{har})\\
& = \tilde k_{min} (j) +\mu(0, \Af_2^{har})+\mu(1, \Af_2^{har} )  \,,
\end{array}
$$
which leads to
\begin{equation}\label{eq:12.10}
  k_{min}(j) =\tilde k_{min} (j)  + 3\,.
\end{equation}
Hence we get back \eqref{eq:9.2}.\\
The general case is easy to treat. For the matrix $  \Af_3^{har\setminus\{j-\ell\}}$ ($\ell=\{1,2,\dots, j-1$), we have for the energy $j$
\begin{equation}
\mu (j,\Af_3^{har}) =\mu(j, \Af_3^{har\setminus\{j-\ell\}} )+ \mu (\ell ,\Af_2^{har})= \mu(j,\Af_3^{har\setminus\{j-\ell\}}) + \ell +1  \,,
\end{equation}
and
\begin{equation}
 k_{min}(j) =  k_{min}(j,   \Af_3^{har\setminus\{j-\ell\}} ) + \sum_{m=0}^{\ell -1} \mu(m, \Af_2^{har})\,.
 \end{equation}

 {\bf Acknowledgements.}\\
 The authors would like to thank Fritz Gesztesy for mentioning to us  the reference \cite{GKZ}.\\

\footnotesize
\bibliographystyle{plain}

\end{document}